\documentclass{amsart}

\usepackage{amssymb, amsthm}%
\usepackage{fncylab}
\usepackage[colorlinks]{hyperref}

\makeatletter
	\newcommand{\numberlike}[2]{\@namedef{c@#1}{\@nameuse{c@#2}}}
\makeatother

\newcommand{\mynewtheorem}[2]{
	\newtheorem{#1}{#2}[section]
	\labelformat{#1}{#2~##1}
	\numberlike{#1}{lemma}
}

\theoremstyle{plain}
\newtheorem{lemma}{Lemma}[section]
\labelformat{lemma}{Lemma~#1}
\mynewtheorem{theorem}{Theorem}
\mynewtheorem{proposition}{Proposition}
\theoremstyle{definition}
\mynewtheorem{definition}{Definition}
\mynewtheorem{remark}{Remark}

\newcommand{\set}[1]{\left\{\,#1\,\right\}}
\newcommand{\with}{\ \vrule\ }
\newcommand{\defa}{:=}

\newcommand{\NN}{\mathbb{N}}
\newcommand{\RR}{\mathbb{R}}
\newcommand{\ZZ}{\mathbb{Z}}

\newcommand{\sat}[1]{\overline{#1}}
\newcommand{\satQl}{\sat{Q}(\la)}

\newcommand{\lcm}[1]{\mathrm{lcm}(#1)}

\newcommand{\QoQ}{\sat{Q}\setminus Q}

\newcommand{\la}{{\boldsymbol{\lambda}}}
\newcommand{\Fla}{F_{\la}}
\newcommand{\sla}[1]{\sigma_{\la}(#1)}
\newcommand{\slap}[1]{\sigma_{\la'}(#1)}

\newcommand{\vv}{\mathbf{v}}
\newcommand{\eb}{\mathbf{e}}
\newcommand{\pb}{\mathbf{p}}
\newcommand{\qb}{\mathbf{q}}
\newcommand{\de}{\boldsymbol{\delta}}
\newcommand{\qq}[1]{``#1''}

\begin{document}
\title{Polytopal affine semigroups with holes deep inside}
\author{Lukas Katth\"an}
\address{Fachbereich Mathematik und Informatik, Philipps-Universit\"at Marburg}%
\email{katthaen@mathematik.uni-marburg.de}%

\date{\today}

\begin{abstract}
Given a non-negative integer $k$, we construct a lattice $3$-simplex $P$ with the following property: The affine semigroup $Q_P$ associated to $P$ is not normal, and every element $q \in \sat{Q}_P \setminus Q_P$ has lattice distance at least $k$ above every facet of $Q_P$.
\end{abstract}
\maketitle

%------------------------------------------------------------------------------

\section{Introduction}
Let $P \subset \RR^{n}$ be a lattice polytope, i.e. a polytope whose vertices have integer coordinates.
We consider the affine semigroup $Q = Q_P \subset \ZZ^{n+1}$ generated by the points $\set{ (p, 1) \in \ZZ^{n+1} \with p \in P \cap \ZZ^{n}}$.
Let $\ZZ Q \subset \ZZ^{n+1}$ be the group generated by the elements of $Q$.
We write $\sat{Q} = \ZZ Q \cap \RR_{\geq 0} Q$ for the normalization of $Q$.
Equivalently, $\sat{Q}$ contains all elements of $\ZZ Q$, such that a positive integral multiple is contained in $Q$.
Then $P$ resp. $Q_P$ are called \emph{normal} if $Q_P = \sat{Q}_P$.
It is a much studied question to characterize normal polytopes.
See \cite{brunsgubel} for background information on affine semigroups and normal polytopes.
The reader should be aware that there is a closely related notion of \emph{integrally closed} polytopes. While a lattice polytope $P$ is called normal if $Q_P = \ZZ Q \cap \RR_{\geq 0} Q$, it is called integrally closed if $Q_P = \ZZ^{n+1} \cap \RR_{\geq 0} Q$. In general, it holds that $\ZZ Q \subseteq \ZZ^{n+1}$, but in many cases of interest one has equality. Therefore, the distinction between normality and integrally closedness is sometimes blurred in the literature.
However, in this paper we will mainly consider the normality of polytopes.

There are results that suggest that the normality of $P$ is somehow determined by the \qq{boundary} of $P$, see for example \cite{longedge}.
Therefore, it seems natural to ask if it is enough to consider normality \qq{near the boundary}.
To make this precise, we give some definitions.
We call an element $q \in \QoQ$ a \emph{hole} in $Q$.
The holes come in families of different dimension, cf. \cite{holes}.
For a facet $F$ of $Q_P$, let $\sigma_F: \ZZ Q_P \rightarrow \ZZ$ be the \emph{lattice height} above $F$, i.e. the linear form with $\sigma_F(F) = 0$ and $\sigma_F(\sat{Q}) = \ZZ_{\geq 0}$, cf. \cite[Remark 1.72]{brunsgubel}.
It is enough to consider elements of lattice height at most $1$ in $\sat{Q}$ to detect families of holes of dimension $n$, see \cite[Exercise 4.15]{brunsgubel}.
The main result of the present note is that this observation does not generalize to higher codimension.
\begin{theorem}\label{thm:main}
	For every natural number $k \in \NN$, there exists a $3$-simplex $P = P(k)$, such that the polytopal affine semigroup $Q_{P}$ is not normal, and every hole $q \in \sat{Q}_{P} \setminus Q_{P}$ has a lattice height of at least $k$ above each facet of $Q_{P}$.
\end{theorem}
In other words, there are polytopes $P$, such that all holes of the semigroup $Q_P$ are \qq{deep inside}.
So it is not sufficient to look for holes near the boundary.
Note that this result is trivial if one considers more general affine semigroups that are not polytopal.
One may just take a big normal polytope $P$ and remove a point from its far interior  to obtain a homogeneous affine semigroups with the desired property.

\section{Rectangular Simplices}
The simplices which we will construct in \ref{thm:main} are special cases of the \emph{rectangular simplices} introduced in \cite{BGrect}.
In this section we recall the construction.
Let $\eb_i \in \RR^{n+1}$ denote the $i$-th unit vector. Let $\la = (\lambda_1, \ldots, \lambda_n)$ be a vector of positive integers. We consider the simplex $\Delta = \Delta(\la) \subset \RR^{n+1}$ with vertices 
\begin{align*}
	 \vv_0 &\defa (0, 0, \dotsc, 0, 1) 			&&= \eb_{n+1}, \\
	 \vv_1 &\defa (\lambda_1, 0, \dotsc, 0, 1) &&= \lambda_1 \eb_1 + \eb_{n+1},\\
	 \vv_2 &\defa (0, \lambda_2, \dotsc, 0, 1) &&= \lambda_2 \eb_2 + \eb_{n+1}, \\
	 & \vdots &&\vdots \\
	 \vv_n &\defa (0, \dotsc, 0, \lambda_n, 1) &&= \lambda_n \eb_n + \eb_{n+1}\,.
\end{align*}
Write $Q = Q(\la)$ for the affine semigroup generated associated to $\Delta(\la)$. Note that $\ZZ Q = \ZZ^{n+1}$, because $\eb_{n+1}, \eb_1 + \eb_{n+1}, \dotsc, \eb_{n} +\eb_{n+1} \in Q$.
There are two kinds of facets of $Q$:
\begin{itemize}
	\item The coordinate hyperplanes are facets of $Q$. We denote  the facet defined by the $i$-th coordinate hyperplane by $F_i$. The lattice height $\sigma_i$ above $F_i$ is given by the $i$-th coordinate of a point $q \in \ZZ Q$.
	\item There is one \qq{skew} facet spanned by the vertices $\lambda_i \eb_i + \eb_{n+1}, 1\leq i \leq n$. Let us denote this facet by $\Fla$. The lattice height above this facet is given by the linear form
	\[ \sla{z} \defa L z_{n+1} - \sum_{i=1}^n \frac{L}{\lambda_i} z_i \,,\]
	where $L \defa \lcm{\lambda_1, \dotsc, \lambda_n}$.
\end{itemize}

\section{Reduction to the skew facet}
\noindent In this section, we prove the following result
that allows us to restrict our attention to the facet $F_\la$.
\begin{proposition}\label{prop:redfacet}
	Let $k$ be a positive integer. Assume that $Q(\la)$ is not normal and every hole has lattice height at least $k$ above $F_\la$.
	Assume further that $Q(\la)$ has no holes in its boundary.
	Then there exists a $\la'$ such that $Q(\la')$ is not normal and its holes have lattice height at least $k$ above every facet.
\end{proposition}
The idea for the proof is taken from \cite[Theorem 1.6]{BGrect}. For a fixed index $1 \leq i \leq n$, set $\ell = \lcm{\lambda_1, \dotsc, \lambda_{i-1}, \lambda_{i+1}, \dotsc, \lambda_n}$. We define $\la' = (\lambda'_1, \dotsc, \lambda'_n)$ by
\[
	\lambda_j' = \begin{cases}
		\lambda_j & \text{ if } j \neq i\,; \\
		\lambda_j + \ell & \text{ if } j = i\;. \\
	\end{cases}
\]
Theorem 1.6 of \cite{BGrect} states that in this situation $Q(\la)$ is normal if and only if $Q(\la')$ is normal. We modify the argument given in \cite{BGrect} to obtain the following result.
\begin{lemma}\label{lem:increase}
	Use the notation as above. Assume that $Q(\la)$ has no holes in its boundary.
	Then there is a bijective linear map $\alpha: \ZZ^{n+1} \rightarrow \ZZ^{n+1}$, such that the preimage of every hole in $Q(\la')$ is a hole in $Q(\la)$ (i.e. $\alpha$ is surjective on holes).
	Moreover, $\alpha$ strictly increases the lattice height of every hole above the facet $F_i$, and it preserves all other lattice heights.
	In particular, $Q(\la')$ also has no holes in its boundary.
\end{lemma}
We can iterate this construction to increase the lattice height of the holes above every facet except $F_\la$. This proves \ref{prop:redfacet}. The map $\alpha$ is taken from the proof of Theorem 1.6 in \cite{BGrect}; we give its definition below.
For the proof of \ref{lem:increase}, we need the following lemma.
\begin{lemma}\label{lem:decrease}
	Let $z \in Q(\la)$, $\tilde{z} \in \ZZ^{n+1}$ with $0 \leq \tilde{z}_i \leq z_i$ for $1 \leq i \leq n$ and  $\tilde{z}_{n+1} = z_{n+1}$. Then $\tilde{z} \in Q(\la)$.
\end{lemma}
\begin{proof}
	We first note that the statement holds if $z_{n+1}=1$.
	This follows from the definition of the simplex $\Delta(\la)$.
	In general, $z$ can be written as a sum of elements of degree $1$. For each summand, we may decrease its components without leaving $Q(\la)$.
	This way, we obtain a representation of $\tilde{z}$ as a sum of degree $1$ elements.
	Hence, $\tilde{z} \in Q(\la)$.
\end{proof}

\begin{proof}[Proof of \ref{lem:increase}]
	Set $L = \lcm{\lambda_1, \dotsc, \lambda_n}$ and $L' = \lcm{\lambda_1', \dotsc, \lambda_n'}$. 
%TODO extra identität
	Note that
	\begin{equation}\label{eq:LLp}
	\frac{L}{\lambda_i} = \frac{\ell}{\gcd(\ell, \lambda_i)} = \frac{\ell}{\gcd(\ell, \lambda_i')} = \frac{L'}{\lambda_i'},
	\end{equation}
	because $\gcd(a,b)=\gcd(a,b+a)$ for all $a,b \in \ZZ$.
	Recall that
	\[ \sla{z} = L z_{n+1} - \sum_{j=1}^n \frac{L}{\lambda_j} z_j \]
	and analogously for $\la'$. We consider the linear form
	\[
		\beta(z) := \frac{\ell}{L}\left(\sla{z}+\frac{L}{\lambda_i}\sigma_i(z)\right) = \ell z_{n+1} - \sum_{\substack{j=1 \\ j\neq i}}^n \frac{\ell}{\lambda_j} z_j \,.
	\]
	defined on $\ZZ^{n+1}$. Note that $\beta$ takes non-negative integer values on $Q(\la)$.
	Using \eqref{eq:LLp}, it is not difficult to verify that 
	\begin{equation}\label{eq:slaslap}
	\slap{z} = \sla{z} +\frac{L}{\lambda_i}\beta(z)
	\end{equation}
	The map $\alpha$ mentioned above can then be defined by $\alpha(z) \defa z + \beta(z) \eb_i$.
	Using \eqref{eq:LLp} and \eqref{eq:slaslap}, one directly verifies that $\slap{\alpha(z)} = \sla{z}$ for every $z \in \ZZ^{n+1}$.
	It follows that $\alpha$ preserves the height above every facet except $F_i$.
	Since $Q(\la)$ has no holes in its boundary, every hole $z$ has $\sla{z} > 0$, so $\beta(z) > 0$ and the height of $\alpha(z)$ above $F_i$ is strictly larger than the height of $z$.
	
	It remains to show that $\alpha$ is surjective on holes. As a preparation, we show that $\alpha(Q(\la)) \subset Q(\la')$.
	We first note that it follows from the discussion above that $\alpha(\sat{Q}(\la)) \subset \sat{Q}(\la')$.
	Next, consider an element $w\in Q(\la)$. It can be written as a sum of elements of degree $1$.
	Since $\alpha$ preserves the degree, this yields a representation of its image $\alpha(w)$ as a sum of degree $1$ elements of $\sat{Q}(\la')$.
	But $Q(\la')$ coincides with $\sat{Q}(\la')$ in degree $1$, hence $\alpha(w) \in Q(\la')$.
		
	Let $z' \in \sat{Q}(\la')\setminus Q(\la')$ be a hole and set $z \defa \alpha^{-1}(z')$. We need to show that $z$ is a hole of $Q(\la)$. It is immediate that $z \notin Q(\la)$, because otherwise $z' = \alpha(z) \in Q(\la')$.
	It remains to show that $z \in \sat{Q}(\la)$, so assume the contrary. Then $z_i < 0$, or equivalently, $z_i' < \beta(z)$. Let $\tilde{z}' \defa z' + (\beta(z)-z_i') \eb_i$, so $\tilde{z}_i = \beta(z)$.
	The linear form $\beta$ does not depend on $z_i$ nor on $\lambda_i$, therefore
	\[
	\beta(z') = \beta(z) = \frac{\ell}{L'}\left(\slap{z'}+\frac{L'}{\lambda_i'}\sigma_i(z')\right)
	\]
	\noindent Using this, we compute
	\begin{align*}
		\slap{\tilde{z}'} &= \slap{z'} + \frac{L'}{\lambda_i'}(z_i' - \beta(z')) \\
		&= \left( \frac{L'}{\ell} - \frac{L'}{\lambda_i'}\right) \beta(z') \\
		&\geq 0
	\end{align*} 
	Here we used that $\lambda_i' = \lambda_i + \ell > \ell$. It follows that $\tilde{z}' \in \sat{Q}(\la')$.
	
	Set $\tilde{z} \defa \alpha^{-1}(\tilde{z}')$. By construction, $\tilde{z}_i = 0$ and $\tilde{z} \in \sat{Q}(\la)$. For this, remember that $\sla{\tilde{z}} = \slap{\tilde{z}'}$.
	But we assumed that $Q(\la)$ has no holes in its boundary, thus $\tilde{z} \in Q(\la)$.
	It follows that $\tilde{z}' = \alpha(\tilde{z}) \in \alpha(Q(\la)) \subset Q(\la')$.
	But now \ref{lem:decrease} implies that $z' \in Q(\la')$, a contradiction.
\end{proof}

\section{Good triples}
In this section, we present our choice of the parameters $\la$.
First, we show that for $3$-dimensional rectangular simplices one of the hypotheses of \ref{prop:redfacet} is always satisfied.
\begin{lemma}
	A $3$-dimensional rectangular simplex $Q(\lambda_1, \lambda_2, \lambda_3)$ has no holes in its boundary.
\end{lemma}
\begin{proof}
	The facets are $2$-dimensional polytopal affine semigroups. Thus, they are normal and even integrally closed in the ambient lattice $\ZZ^4$ (cf. \cite[Corollary 2.54]{brunsgubel}). Hence, $Q(\lambda_1, \lambda_2, \lambda_3)$ has no holes in its boundary.
\end{proof}
\noindent It is now sufficient to find $(\lambda_1, \lambda_2, \lambda_3)$ such that the distance of the holes to the facet $\Fla$ is bounded below. This is achieved with the following class of triples.
\begin{definition}
	Let $\lambda_1 \leq \lambda_2 \leq \lambda_3$ be positive integers and let $\de \defa (-1,2,-1,0) \in \ZZ^4$. We call $\la = (\lambda_1, \lambda_2, \lambda_3)$ a \emph{good triple} if the following conditions are met: 
	\begin{enumerate}
		\item $\lambda_1, \lambda_2$ and $\lambda_3$ are pairwise coprime;
		\item $\sla{\de} = 2$, i.e. $\lambda_2 \lambda_3 - 2 \lambda_1 \lambda_3 + \lambda_1 \lambda_2 = 2$;
		\item $\lambda_1 +2 < \lambda_2$.
	\end{enumerate}
\end{definition}
\noindent The following can be verified directly.
\begin{proposition}\label{r:exgoodtriple}
	Let $\lambda_1 \geq 5$ be an odd positive integer. Then $(\lambda_1, 2 \lambda_1 - 1, 2 \lambda_1^2-\lambda_1-2)$ is a good triple.
\end{proposition}
\noindent Next, we show that good triples yield examples of simplices satisfying our need. So the next proposition completes the proof of \ref{thm:main}.
\begin{proposition}\label{r:goodtriple}
	Let $\la = (\lambda_1, \lambda_2, \lambda_3)$ be a good triple. Then $Q(\la)$ is not normal and every hole has lattice distance at least $\lambda_1 + 2$ over $\Fla$.
\end{proposition}
\noindent We prepare two lemmata before we prove this proposition.
\begin{lemma}\label{lem:eind}
	Let $\lambda_1, \dotsc, \lambda_n$ be pairwise coprime. For every positive integer $s > 0$, there exists at most one element $\qb \in \satQl$ with $\sla{\qb} = s$ and $\sigma_i(\qb) < \lambda_i$ for every $1 \leq i \leq n$.
\end{lemma}
\begin{proof}
	 This follows easily from the observation that $\ker \sigma_\la$ is generated as a group by $\vv_1, \dotsc, \vv_n$.
\end{proof}
We note that the proof of \ref{lem:eind} is inspired by the proof of Proposition 1.3 in \cite{BGrect}.
\begin{lemma}\label{lem:1nor}
	Let $\lambda_1, \dotsc, \lambda_n$ be pairwise coprime and let $s$ be a positive integer.
	Assume that for every positive integer $t \leq s$, there exists an element $\pb_t \in Q(\la)$ with $\sla{\pb_t} = t$ and $\sigma_i(\pb_t) < \lambda_i$ for every $i$.
	Then every hole $\qb \in \satQl \setminus Q(\la)$ has $\sla{\qb} > s$.
\end{lemma}
\begin{proof}
	We may assume that $\sigma_i(\qb) < \lambda_i$ for every $i$, because otherwise we can subtract $\vv_i$.
	Now the claim is immediate from the preceding \ref{lem:eind}.
\end{proof}
\begin{proof}[Proof of \ref{r:goodtriple}]
	First, we show that both $\lambda_1$ and $\lambda_3$ are odd.
	For this assume to the contrary that $\lambda_1 = 2 \lambda_1'$ for an integer $\lambda_1'$.
	Then $\lambda_2 \lambda_3 = 2(1+2\lambda_1' \lambda_3 - \lambda_1' \lambda_2)$, thus either $\lambda_2$ or $\lambda_3$ are even, violating the coprimeness assumption.
	The proof that $\lambda_3$ is odd is analogous.

	Next, consider the vector 
	\[
	\pb \defa \frac{1}{2}(\vv_1 + \vv_3 + \de)
	  = \left(\frac{\lambda_1-1}{2},\, 1,\, \frac{\lambda_3 - 1}{2},\, 1 \right) .
	\]
	It follows from $\lambda_1$ and $\lambda_3$ odd and $\sla{\de} = 2$ that $\pb \in Q(\la)$ and  $\sla{\pb} = 1$.

	For $0 \leq k \leq \frac{\lambda_1 - 1}{2}$ it holds that $\pb + k \de \in Q(\la)$ and $\sla{\pb+k\de} = 1 + 2k$.
	Moreover, $\sigma_i(\pb+k\de) \leq \lambda_i$ for $i=1,2,3$.
	Further, it holds that $2 \pb + k \de \in Q(\la)$, $\sla{2\pb+k\de} = 2 + 2k$ and $\sigma_i(2\pb+k\de) \leq \lambda_i$ for $i=1,2,3$. For the last statement with $i=2$, we use that $\lambda_1+1 < \lambda_2$. Thus, we apply \ref{lem:1nor} with the vectors $\pb+k \de$ and $2 \pb + k \de$ for $0 \leq k \leq \frac{\lambda_1 - 1}{2}$ to conclude that there exists no hole with lattice height strictly less than $\lambda_1 + 2$ above $\Fla$.

	\noindent Let
	\[
	 \qb \defa \pb + \left(\frac{\lambda_1-1}{2}+1\right)\de + \vv_1
	   = (\lambda_1 - 1,\, \lambda_1 + 2,\, \frac{\lambda_3 - \lambda_1}{2} - 1 ,\, 2) \,.
	\]
	The components of $\qb$ are non-negative integers and $\sla{\qb} = \lambda_1 + 2$, hence $\qb \in \sat{Q}(\la)$. We claim that $\qb \notin Q(\la)$. This clearly implies that $Q(\la)$ is not normal.
	So assume that $\qb = \qb_1 + \qb_2$ for $\qb_1, \qb_2 \in Q(\la)$. Since $\lambda_1, \lambda_2$ and $\lambda_3$ are pairwise coprime, the only elements of $Q(\la)$ in $\Fla$ of degree $1$ are $\vv_1, \vv_2$ and $\vv_3$.
	But $\lambda_1 -1 < \lambda_1$, $\lambda_1 + 2 < \lambda_2$ (by assumption) and $\frac{\lambda_3 - \lambda_1}{2} - 1 < \lambda_3$, so $\qb - \vv_i$ has a negative component for $i=1,2,3$.
	It follows that $\sla{\qb_1}, \sla{\qb_2} > 0$. Since $\sla{\qb} = \lambda_1 + 2$ is odd, one of $\sla{\qb_1}$ and  $\sla{\qb_2}$ is odd, too, say $\sla{\qb_1}$. 
	By \ref{lem:eind}, all elements $\vv$ of $Q(\la)$ of degree $1$ with $\sla{\vv} \leq \lambda_1$ and $\sla{\vv}$ odd are of the form $\pb + k \de$ for $0 \leq k \leq \frac{\lambda_1 - 1}{2}$.
	But $\qb - (\pb + k \de) = \vv_1 + \left(\frac{\lambda_1 - 1}{2} +1 - k\right) \de$ has a negative third component.
	Thus $\qb$ cannot be written as a sum of elements of degree $1$ in $Q(\la)$.
\end{proof}

\bibliographystyle{abbrv}
\bibliography{RandAbstand}
\end{document}